\documentclass{amsart}

\usepackage{amsmath,amsthm,amssymb,amscd}

\newtheorem{prop}{Proposition}[section]
\newtheorem{theorem}[prop]{Theorem}
\newtheorem{lemma}[prop]{Lemma}
\newtheorem{corr}[prop]{Corollary}
\theoremstyle{definition}
\newtheorem{definition}[prop]{Definition}
\theoremstyle{remark}
\newtheorem{remark}[prop]{Remark}
\newtheorem{example}[prop]{Example}

\DeclareMathOperator{\Prim}{Prim}

\DeclareMathOperator{\supp}{supp}
\DeclareMathOperator{\ran}{ran}

\DeclareMathOperator{\Ind}{Ind}
\DeclareMathOperator{\Ex}{Ex}

\newcommand{\mcal}[1]{\mathcal{#1}}

\newcommand{\lset}[1]{{}_{#1}}
\newcommand{\inv}{^{-1}}
\newcommand{\llangle}{\langle\!\langle}
\newcommand{\rrangle}{\rangle\!\rangle}
\newcommand{\spn}{\text{span}}
\newcommand{\cspn}{\overline{\text{span}}}
\newcommand{\sidehat}{^{\wedge}}

\newcommand{\unit}{^{(0)}}

\allowdisplaybreaks

\title[The Mackey Machine for Groupoid Crossed Products]{The Mackey Machine
  for Crossed Products by Regular Groupoids. I}
\author{Geoff Goehle}
\address{Mathematics and Computer Science Department, Stillwell 426,
  Western Carolina University, Cullowhee, NC 28723}
\email{grgoehle@email.wcu.edu}

\begin{document}

\begin{abstract}
We first describe a Rieffel induction system for groupoid crossed
products. We then use this induction system to show that, given a
regular groupoid $G$ and a dynamical system $(A,G,\alpha)$, every
irreducible representation of $A\rtimes G$ is induced from a
representation of the group crossed product $A(u)\rtimes S_u$ where
$u\in G\unit$, $A(u)$ is a fibre of $A$, and $S_u$ is a stabilizer
subgroup of $G$. 
\end{abstract}

\maketitle

\section*{Introduction}

The term ``Mackey Machine'' generally refers to a program for
extracting information about the representation theory of a group, group
action, or dynamical system using the representation theory of certain
subsystems as well as some sort of process for inducing
representations.   The goal of such a program is to, under suitable
hypotheses, identify the topology on the primitive ideal space, or
spectrum, of an associated universal $C^*$-algebra (e.g. the group
$C^*$-algebra).  These notions have been applied quite successfully in many
different contexts.  There is the original work done by Mackey
\cite{mackey1, mackey2} and then later generalized by Fell and Doran
\cite{felldoran}, as well as Takesaki's research \cite{takesaki}
on group crossed products.  In terms of methodology, however, our approach will
be closer to that of Rieffel, who recast Mackey's theory using Hilbert
modules and Morita equivalence \cite{rieffelinduce, rieffelunit}.  This
technique was then applied to crossed products by Green \cite{lscovalg} and
Echterhoff \cite{ectcrossed}.  More recently, work has continued on 
the Mackey Machine for group crossed products in
\cite{inducprimide} as well as for groupoid $C^*$-algebras in
\cite{irredreps,geneffhan}.   In this paper, however, we will develop 
the Mackey Machine for groupoid crossed products.  

In general, when implementing the Mackey Machine for a
groupoid dynamical system $(A,G,\alpha)$ one must choose between two
natural classes of subgroups.  The first is the set of
stabilizer subgroups for the action of the groupoid $G$ on its unit
space $G\unit$.  The second is the set of stabilizer subgroups for the
action of the groupoid $G$ on the primitive ideal space of the
$C^*$-algebra $A$.  This series studies the irreducible representations
of $A\rtimes G$ using the {\em former} class, and in particular using the
stabilizer subgroupoid $S$ of $G$.  Our eventual goal will be to
use induction to identify the spectrum of $A\rtimes_\alpha G$ as a 
quotient of the spectrum of $A\rtimes_{\alpha} S$.  This falls
under the ideology of the Mackey Machine because, as a
set, $(A\rtimes_{\alpha} S)\sidehat$ is equal to the disjoint union
$\coprod A(u)\rtimes S_u$, where the $A(u)$ are the fibres of $A$ and
the $S_u$ are the stabilizer subgroups associated to the action of $G$
on its unit space.  Our first step towards this goal will be to
demonstrate an induction technique for groupoid crossed products.  The
construction of induced representations contained in this paper 
is very general and
extends ``Rieffel''-type induction for groups, groupoids, 
and group crossed products \cite{rieffelinduce,irredreps,lscovalg}.
After induction, the next step in developing the Mackey Machine is to
prove that every irreducible representation of the system is induced
from a representation of a substructure.  Specifically we
show that every irreducible representation of $A\rtimes
G$ is induced from a representation of a ``stabilizer'' crossed
product $A(u)\rtimes S_u$ for some $u\in G\unit$.  This result is
similar in spirit to the Gootman-Rosenberg-Sauvageot Theorem or the
results of \cite{geneffhan}, but is more closely related to
\cite{cpsmooth} since we make use of a regularity hypothesis.
The structure of this paper is roughly as follows. Section
\ref{sec:group-cross-prod} covers some basic crossed product theory
followed by the induction results in Section \ref{sec:induction}.  Section
\ref{sec:basic-constructions} contains some basic constructions which
will be needed in Section \ref{sec:stab-t_0-orbits} where we prove the
surjectivity result described above.  

Before we begin in earnest we should first make some remarks about our
hypotheses.  Because we are looking to identify the spectrum of a
$C^*$-algebra we will need some sort of ``Type I'' or ``regularity''
condition.  Specifically, we will need to assume that the groupoid $G$
satisfies the conditions of the Mackey-Glimm Dichotomy
\cite{groupoiddichotomy}.  We will refer to such groupoids as {\em
  regular} groupoids.  
Of course, the assumption that $G$ is regular could be easily replaced
by any of the fourteen equivalent conditions in the Mackey-Glimm
Dichotomy, including assuming that $G\unit/G$ is $T_0$ or that $G\unit/G$ is
almost Hausdorff.  Finally, it should be noted that the results of
this paper are contained, with more detail and a great deal of
background material, in the author's thesis \cite{mythesis}.

\section{Groupoid Crossed Products}
\label{sec:group-cross-prod}

Throughout the paper we will let $G$ denote a second countable,
locally compact Hausdorff groupoid with a Haar system $\lambda$.
Given an element $u\in G\unit$ of the unit space of $G$ we
will use $S_u=\{\gamma\in G : s(\gamma)=r(\gamma)=u\}$ 
to denote the stabilizer, or isotropy, subgroup of $G$
over $u$.  Since groupoids act on fibred objects, we must consider the
theory of $C_0(X)$-algebras and upper-semicontinuous $C^*$-bundles
(usc-bundles) if we are to work with groupoid dynamical systems.  It
is assumed that the reader is familiar with $C_0(X)$-algebras and
their associated bundles, although either \cite[Appendix C]{tfb2} or
\cite[Section 3.1]{mythesis} will serve as a reference.  
Throughout the rest of the paper we will let $A$ denote a separable
$C_0(G\unit)$-algebra and will let $\mcal{A}$ be its associated
usc-bundle.  We use $\Gamma_0(G\unit,\mcal{A})$ to denote the
space of sections of $\mcal{A}$ which vanish at infinity and identify
this space with $A$ in the usual fashion.  

\begin{remark}
Because it will be fundamental to what follows, let us recall the
``usual'' method of identifying a $C_0(X)$-algebra $A$ with
$\Gamma_0(X,\mcal{A})$.  Given $x\in X$ we let 
\[
I_x = \cspn\{\phi\cdot a : \phi\in C_0(X), a\in A, \phi(x)=0\}
\]
and define $A(x):= A/I_x$.  We denote the image of $a$ in $A(x)$ by
$a(x)$.  Then \cite[Theorem C.26]{tfb2} implies that $\mcal{A}=
\coprod_{x\in X} A(x)$ has an usc-bundle structure and that the map
sending $a$ to $x\mapsto a(x)$ is an isomorphism of $A$ onto
$\Gamma_0(X,\mcal{A})$.  
\end{remark}

We take this opportunity to establish some basic facts about
$C_0(X)$-algebras which we shall need in the next section. Given a
$C_0(X)$-algebra $A$, its associated usc-bundle $p:\mcal{A}\rightarrow
X$ and $Y$ a locally compact Hausdorff subset of $X$, we define
$\mcal{A}|_Y:=p\inv(Y)$.  Furthermore, we use $A(Y)$ to denote
$\Gamma_0(Y,\mcal{A}|_Y)$, which we shall identify with 
$\Gamma_0(Y,\mcal{A})$.   
Next, if $\tau:Y\rightarrow X$ is a continuous map then we let $\tau^*
\mcal{A}$ denote the pullback of $\mcal{A}$ by $\tau$.  In other words
$\tau^*\mcal{A} = \{(y,a)\in Y\times\mcal{A}:\tau(y) = p(a)\}$.  
In a slight
break from the standard notation, we will use $\tau^*A$ to denote the
section algebra $\Gamma_0(Y,\tau^*\mcal{A})$.  Of course, when $\tau$ is
surjective this can be identified with $C_0(Y)\otimes_{C_0(X)} A$,
which is the the usual notion of the
pull-back of a $C_0(X)$-algebra \cite[Proposition 1.3]{pullback}.
When $\tau$ is not surjective, $\tau^*A$ can be identified with $C_0(Y)\otimes_{C_0(\ran\tau)} A(\ran \tau)$.  

\begin{remark}
An element $f \in \tau^*\mcal{A}$ is technically of the form $f(y) =
(y,\tilde{f}(y))$ where $\tilde{f}$ is a continuous map from $Y$ into
$\mcal{A}$ such that $p(\tilde{f}(y))=\tau(y)$ for all $y\in Y$.
However, we rarely make a distinction between $f$ and $\tilde{f}$
and often confuse the two. 
\end{remark}

One of the important facts about pull backs is the following

\begin{prop}
\label{prop:9}
Suppose $X$ and $Y$ are locally compact Hausdorff spaces, $A$ is a
$C_0(X)$-algebra, and $\tau:Y\rightarrow X$ is a continuous map.
Given $f\in C_c(Y)$ and $a\in A$ define $f\otimes a(y) =
f(y)a(\tau(y))$ for all $y\in Y$.  Then $f\otimes a\in
\Gamma_c(Y,\tau^*\mcal{A})$ and 
\[
C_c(Y)\odot A := \spn\{f\otimes a : f\in C_c(Y), a\in A\}
\]
is dense in $\tau^*A$.  Whats more, $C_c(Y)\odot A$ is dense in
$\Gamma_c(Y,\tau^*\mcal{A})$ with respect to the inductive limit
topology. 
\end{prop}

\begin{proof}
The statements about $f\otimes a$ are trivial, and it follows in a
straightforward manner from \cite[Proposition C.24]{tfb2} that
$C_c(Y)\odot A$ is dense in $\tau^*A$.  Given $f\in
\Gamma_c(Y,\tau^*\mcal{A})$ recall that a sequence $f_i$ converges to $f$
with respect to the inductive limit topology if it converges to $f$
uniformly and the $\supp f_i$ are all eventually contained in some
compact set $K$.  We know that there exists $f_i\in C_c(Y)\odot A$
such that $f_i\rightarrow f$ uniformly.  Since $C_c(Y)\odot A$ is
closed under the action of $C_0(Y)$, it is a simple matter to
cut down the $f_i$ so that they are supported on a compact
neighborhood of $\supp f$.  
\end{proof}

Now, given $A$ and $G$ as above
we let $\alpha$ denote an action of $G$ on $A$ as defined in
\cite[Definition 4.1]{renaultequiv}.  Recall that this means $\alpha$
is given by a collection of isomorphisms $\{\alpha_\gamma\}_{\gamma\in
  G}$ such that $\alpha_\gamma :A(s(\gamma))\rightarrow A(r(\gamma))$,
$\alpha_{\gamma\eta} = \alpha_\gamma\circ\alpha_\eta$ whenever
$s(\gamma)=r(\eta)$, and that $\gamma\cdot a \mapsto \alpha_\gamma(a)$
is a continuous action of $G$ on $\mcal{A}$.  Given a groupoid
dynamical system $(A,G,\alpha)$ consisting of $A$, $G$, and $\alpha$ as
above, we 
construct the groupoid crossed product $A\rtimes_\alpha G$ as in
\cite[Section 4]{renaultequiv} or \cite[Chapter 3]{mythesis}.  Without
going into details we will mention that $A\rtimes G$ is a universal
completion of the algebra of compactly supported sections
$\Gamma_c(G,r^*\mcal{A})$, where $r^*\mcal{A}$ denotes the pull back
of $\mcal{A}$ by the range map, given the convolution and involution
operations
\[
f*g(\gamma) = \int_G f(\eta)\alpha_\eta(g(\eta\inv\gamma))
d\lambda^{r(\gamma)}(\eta) \quad\text{and}\quad
f^*(\gamma) = \alpha_\gamma(g(\gamma\inv)^*).
\]

What is, perhaps, more motivating than the details of the construction
of the crossed product is the following example.  

\begin{example}
\label{ex:haus}
Suppose $A$ is a $C^*$-algebra with Hausdorff spectrum $\widehat{A}$
and recall that we can view $A$ as a $C_0(\widehat{A})$-algebra with
fibres $A/\ker\pi$.  Let $\mcal{A}$ be the corresponding usc-bundle.
Given a locally compact Hausdorff group $H$ and an action $\alpha$ of
$H$ on $A$ recall that there is an induced action of $H$ on
$\widehat{A}$ given by $s\cdot \pi = \pi\circ\alpha_s\inv$.  Let $G$
be the corresponding transformation groupoid $G = H\times
\widehat{A}$.  Now, given $(s,\pi)\in G$ define $\beta_{(s,\pi)}:
A(s\inv\cdot \pi)\rightarrow A(\pi)$ by 
\[
\beta_{(s,\pi)}(a(s\inv\cdot \pi)) = \alpha_s(a)(\pi).
\]

A straightforward computation shows that $\beta_{(s,\pi)}$ is a well
defined isomorphism.  It is similarly easy to prove that
$\beta_{(s,\pi)(t,\rho)} = \beta_{(s,\pi)}\circ\beta_{(t,\rho)}$
whenever $s\inv\cdot \pi = \rho$.  The last thing we need to do
to show that $\beta$ is a groupoid action is to prove it is
continuous.  This is, again, straightforward as long as you apply
\cite[Proposition C.20]{tfb2}.  Since $(A,G,\beta)$ is a groupoid
dynamical system, we can form the crossed product $A\rtimes_\beta G$.
We claim that this is naturally isomorphic to the group crossed
product $A\rtimes_\alpha H$.  First, recall that $A\rtimes_\alpha H$ is
a universal completion of the function algebra $C_c(H,A)$.  Now, given
$f\in \Gamma_c(G,r^*\mcal{A})$ define $\Phi(f)(s)(\pi) = f(s,\pi)$
for $s\in H$ and $\pi\in \widehat{A}$.  It follows quickly that
$\Phi(f)(s)$ defines an element of $A$ and that $\Phi(f)$ is a
continuous compactly supported function on $H$.  Thus
$\Phi:\Gamma_c(G,r^*\mcal{A})\rightarrow C_c(H,A)$ and simple
calculations show that $\Phi$ is a $*$-homomorphism which is
continuous with respect to the inductive limit topology.  It is now an
immediate result of Renault's Disintegration Theorem \cite[Theorem
7.8,7.12]{renaultequiv} that $\Phi$ extends to a
homomorphism from $A\rtimes G$ into $A\rtimes H$.  The goal is to show
that this extension is an isomorphism.  

Let $D = \Gamma_c(\widehat{A},\mcal{A})$ and observe that $D$ is dense
in $A$.  Consider the set of sums of elementary tensors $C_c(H)\odot
D$ and recall \cite[Lemma 1.87]{tfb2} that this set is
dense in $C_c(H,A)$ with respect to the inductive limit
topology. Given $\phi\otimes a \in C_c(H)\odot D$  define $f(s,\pi) =
\phi(s)a(\pi)$ and observe that $f\in \Gamma_c(G,r^*\mcal{A})$.
Clearly $\Phi(f) = \phi\otimes a$ so that $C_c(H)\odot D \subset \ran
\Phi$.  Hence $\ran\Phi$ is dense with respect to the
inductive limit topology and therefore with respect to the $I$-norm.  
In order to show that $\Phi$ is
an isomorphism we exhibit the following inverse.  First, observe that $\Phi$
is bijective on $\Gamma_c(G,r^*\mcal{A})$ so that we can define an
inverse map $\Psi =
\Phi\inv:\ran\Phi\rightarrow\Gamma_c(G,r^*\mcal{A})$ 
by $\Psi(f)(s,\pi) = f(s)(\pi)$.  
It is easy to show that $\Psi$ is a $*$-homomorphism.
Straightforward calculations then show that $\Psi$ is $I$-norm decreasing.
Since $\ran\Phi$ is dense in $C_c(H,A)$ with respect to the $I$-norm,
we can extend $\Psi$ to $C_c(H,A)$.  This extension is still an
$I$-norm decreasing $*$-homomorphism on $C_c(H,A)$ and it follows from
\cite[Corollary 2.47]{tfb2} that $\Psi$ extends to all of $A\rtimes_\alpha
H$.  Finally, since $\Psi$ and $\Phi$ are inverses on dense subsets,
it follows that they are inverses everywhere and that $A\rtimes H$ is
isomorphic to $A\rtimes G$. 
\end{example}

The reason we went through the trouble of describing this example in detail is
that it is useful to think of groupoid crossed products as
generalizing group crossed products in this manner.  For instance,
in the case where $A$ has Hausdorff spectrum the results of
Section \ref{sec:stab-t_0-orbits} can be seen as generalizing
corresponding results for group crossed products when the
generalization is done via Example \ref{ex:haus}.  
On the other hand, when we view groups
as groupoids with a single unit the results of Section
\ref{sec:stab-t_0-orbits} become trivial.

\section{Induction}
\label{sec:induction}

The goal for this section is to show that if we are given a groupoid
dynamical system $(A,G,\alpha)$ and a closed subgroupoid $H$ of $G$
with its own Haar system then there is an induction process which
takes representations of $A\rtimes_\alpha H$ and creates
representations of $A\rtimes_\alpha G$.  This induction process has
its roots in the Rieffel correspondence and it is assumed that the
reader is familiar with representations induced via imprimitivity
bimodules as described in \cite[Chapters 2,3]{tfb}.  We will apply
this theory to groupoid crossed products by constructing a right
Hilbert $A\rtimes_\alpha H$-module $\mcal{Z}_H^G$ and proving the
following theorem. 

\begin{theorem}
\label{thm:induce}
Suppose $(A,G,\alpha)$ is a groupoid dynamical system and
that $H$ is a closed subgroupoid of $G$ with a Haar system.  Then
given a representation $R$ of $A(H\unit)\rtimes_{\alpha|_H} H$ on
$\mcal{H}$ we may form the induced representation $\Ind_H^G R$ of
$A\rtimes_\alpha G$ on $\mcal{Z}_H^G \otimes_{A\rtimes H} \mcal{H}$
which is defined for $f\in \Gamma_c(G,r^*\mcal{A})$, $z\in\mcal{Z}_0$
and $h\in\mcal{H}$ by 
\[
\Ind_H^G R(f)(z\otimes h) = f\cdot z \otimes h
\]
where 
\begin{equation}
\label{eq:1}
f\cdot z (\gamma) = \int_G \alpha_\gamma\inv(f(\eta))z(\eta\inv\gamma)
d\lambda^{r(\gamma)}(\eta).
\end{equation}
\end{theorem}

Before we prove this theorem we must construct the Hilbert module
$\mcal{Z}_H^G$, which we will do by using Renault's Equivalence
Theorem to build an imprimitivity bimodule.  First, 
however, we need to construct a new action from $\alpha$
and to do this we must back up even further and discuss
groupoid equivalence and imprimitivity groupoids.  We assume that the
reader is familiar with these notions, although \cite{groupoidequiv}
and \cite{tfb} will serve as references.  

\begin{definition}
Suppose $G$ is a locally compact Hausdorff groupoid
and $H$ is a closed subgroupoid with a Haar
system.  Let $X = s\inv(H\unit)$.  
We define the {\em imprimitivity groupoid} $G^H$ to be the
quotient of $X*X = \{(\gamma,\eta) : s(\gamma) =
s(\eta)\}$ by $H$ where $H$ acts diagonally via right translation.  We
give $G^H$ a groupoid structure with the operations
\[
[\gamma,\eta][\eta,\zeta] = [\gamma,\zeta],\quad\text{and}\quad
[\gamma,\eta]\inv = [\eta,\gamma].
\]
\end{definition}

The key facts we need concerning these objects are contained in the
following

\begin{prop}
\label{prop:3}
Suppose $G$ is a locally compact Hausdorff groupoid
and $H$ is a closed subgroupoid with a Haar
system.  Then $G^H$ is a locally compact Hausdorff
groupoid with unit space $X/H$ and Haar system 
$\{\mu^{[\xi]}\}$ defined for $f\in C_c(G^H)$ by 
\begin{equation}
\label{eq:2}
\int_{G^H} f([\zeta,\eta])d\mu^{[\xi]}([\zeta,\eta]) = 
\int_G f([\xi,\eta])d\lambda_{s(\xi)}(\eta).
\end{equation}
Furthermore, $X$ is a $(G^H,H)$-equivalence where $H$ acts on $X$ by
right translation and $G^H$ acts via the operation 
$[\gamma,\eta]\cdot \eta = \gamma$.  
\end{prop}

\begin{proof}
This material is all known.  The construction of the imprimitivity
groupoid and the results concerning the $(G^H,H)$-equivalence can be
found in \cite[Section 2]{groupoidequiv}.  The fact that $\mu$ is a
Haar system for $G^H$ can be found in \cite[Section 2]{irredreps}.  It
is worth observing that in order for $\mu$ to exist all we really
require is that $H$ have open range and source.  
\end{proof}

Once we have our imprimitivity groupoid we can couple it with the
action of $G$ on $A$ to form a new dynamical system. 

\begin{prop}
\label{prop:2}
Let $(A,G,\alpha)$ be a groupoid dynamical system and suppose $H$ is a
closed subgroupoid of $G$ with a Haar system.  Let $X=s\inv(H\unit)$,
$G^H$ be the imprimitivity groupoid, define $\rho : X/H\rightarrow
G\unit$ by $\rho([\gamma]) = r(\gamma)$, and let $\rho^*A$
be the pull back of $A$.  Then the collection $\sigma =
\{\sigma_{[\gamma,\eta]}\}_{[\gamma,\eta]\in G^H}$ where
$\sigma_{[\gamma,\eta]}:A(r(\eta)) \rightarrow A(r(\gamma))$
is given for $a\in A(r(\eta))$ by 
\begin{equation}
\label{eq:3}
\sigma_{[\gamma,\eta]}(a) = \alpha_{\gamma\eta\inv}(a)
\end{equation}
defines an action of $G^H$ on $\rho^*A$. 
\end{prop}

\begin{proof}
First, recall that we identify the unit space of $G^H$ with $X/H$ so
that we may view the pull back $\rho^*A$ as a
$C_0((G^H)\unit)$-algebra.  Furthermore, given $[\gamma]\in X/H$
the fibre of $\rho^*A$ over $[\gamma]$ is  $\rho^* A([\gamma]) =
A(\rho([\gamma])) = A(r(\gamma))$.  Thus $\sigma_{[\gamma,\eta]}$ can
be viewed as an isomorphism of $\rho^*A([\eta])$ onto
$\rho^*A([\gamma])$.  It is now straightforward to see that $\sigma$ is
well defined, respects the groupoid operations, and is continuous. 
\end{proof}

\begin{remark}
Since $(\rho^*A,G^H,\sigma)$ is a groupoid dynamical system, we may
form the crossed product $\rho^*A\rtimes_\sigma G^H$ as the completion
of $\Gamma_c(G^H,r^*\rho^*\mcal{A})$.  Elements of this function
algebra can be viewed as continuous, compactly supported maps from
$G^H$ into $\mcal{A}$ such that $f([\gamma,\eta])\in A(r(\gamma))$ for
all $[\gamma,\eta]\in G^H$. 
\end{remark}

We can now use $\sigma$ to build the promised imprimitivity bimodule.

\begin{prop}
\label{prop:1}
Suppose $(A,G,\alpha)$ is a groupoid dynamical system.  
Furthermore, suppose $H$ is a
closed subgroupoid with Haar system $\lambda_H$.  Let $X$, $G^H$,
$\rho$ and $\sigma$ be as in Proposition \ref{prop:2}. 
Then $\mcal{Z}_0 = \Gamma_c(X,s^*\mcal{A})$ becomes a
pre-$\rho^*A\rtimes_\sigma G^H - A(H\unit)\rtimes_{\alpha|_H}
H$-imprimitivity bimodule with respect to the following actions for
$f\in \Gamma_c(G^H,r^*\rho^*\mcal{A})$, $g\in
\Gamma_c(H,r^*\mcal{A})$, and $z,w\in \mcal{Z}_0$: 
\begin{align*}
f\cdot z(\gamma) &= \int_G
\alpha_\gamma\inv(f([\gamma,\eta]))z(\eta)d\lambda_{s(\gamma)}(\eta)
\\
z\cdot g(\gamma) &= \int_H
\alpha_\eta(z(\gamma\eta)g(\eta\inv))d\lambda_H^{s(\gamma)}(\eta) \\
\llangle z,w \rrangle_{A\rtimes H}(\gamma) &= \int_G
z(\eta\gamma\inv)^* \alpha_\gamma(w(\eta))d\lambda_{s(\gamma)}(\eta)
\\
\lset{A\rtimes G^H}\llangle z,w \rrangle([\gamma,\eta]) &= \int_H
\alpha_{\gamma\xi}(z(\gamma \xi)w(\eta\xi)^*)d\lambda_H^{s(\gamma)}(\xi)
\end{align*}
The completion of $\mcal{Z}_0$, denoted $\mcal{Z}_H^G$, is a
$\rho^*A\rtimes_\sigma G^H - A(H\unit)\rtimes_{\alpha|_H}
H$-imprimitivity bimodule and $\rho^*A\rtimes_\sigma G^H$ is Morita
equivalent to $A(H\unit)\rtimes_\alpha H$. 
\end{prop}

\begin{proof}
We will make use of the machinery of equivalence bundles from
\cite{renaultequiv,frenchrenault}.  In particular we are going to build an
equivalence bundle \cite[Definition 5.1]{renaultequiv} and then use Renault's
Equivalence Theorem \cite[Theorem 5.5]{renaultequiv}.  
First, recall from Proposition \ref{prop:3} that $X$ is
a $(G^H,H)$-equivalence.  The source map for $X$, $s_X:X\rightarrow
H\unit$,  is just the
restriction of the source map on $G$ to $X$, and the range map,
$r_X:X\rightarrow X/H$,
is the quotient map.  Our equivalence bundle is then the pull back 
$\Upsilon = s_X^* \mcal{A}$.  This is an usc-bundle
over $X$ and we can identify the fibre over $\gamma\in X$ with
$A(s(\gamma))$.  This allows us to give $\Upsilon_{\gamma}$ an
$A(r(\gamma))-A(s(\gamma))$-imprimitivity bimodule structure via the
isomorphism $\alpha_\gamma$ in the usual fashion \cite[Example
3.14]{tfb}.  Next, we define actions of $G^H$ and $H$ on $\Upsilon$
for $[\eta,\zeta]\in G^H$, $\xi\in H$ and $(\gamma,a)\in \Upsilon$ by 
\[
(\gamma,a)\cdot \xi := (\gamma\xi, \alpha_\xi\inv(a)),\quad\text{and}\quad
[\eta,\zeta]\cdot (\gamma,a) := (\eta\gamma\zeta\inv,a).
\]
It is then straightforward to show that $\Upsilon$ satisfies the
definition of an equivalence bundle and it follows from Renault's
Equivalence Theorem that $\mcal{Z}_0 = \Gamma_c(X,s^*A)$ is a
pre-$\rho^*A\rtimes_\sigma G^H-A(H\unit)\rtimes_\alpha
H$-imprimitivity bimodule.  It is then a matter of calculating that
the operations on $\mcal{Z}_0$ all have the correct form and the rest
of the proposition follows. 
\end{proof}

Now that we have built the imprimitivity bimodule $\mcal{Z}_H^G$
mentioned in Theorem \ref{thm:induce} we need to let $A\rtimes_\alpha
G$ act as adjointable operators.  However, it follows from the
general theory of imprimitivity bimodules that this is equivalent to
letting $A\rtimes_\alpha G$ map into the multipliers on $\rho^*
A\rtimes_\sigma G^H$.  Theorem \ref{thm:induce} will then follow
immediately from the following

\begin{prop}
\label{prop:4}
Let $(A,G,\alpha)$ be a groupoid dynamical system, $H$ a closed
subgroupoid of $G$ with a Haar system and $\mcal{Z}_H^G$ the
associated imprimitivity bimodule completed from $\mcal{Z}_0$ as in
Proposition \ref{prop:1}.
Then there is a nondegenerate homomorphism $\phi:A\rtimes_\alpha G
\rightarrow \mcal{L}(\mcal{Z}_H^G)$ such that for $f\in
\Gamma_c(G,r^*\mcal{A})$ and $z\in \mcal{Z}_0$
\begin{equation}
\label{eq:4}
\phi(f)z(\gamma) = \int_G
\alpha_\gamma\inv(f(\eta))z(\eta\inv\gamma)d\lambda^{r(\gamma)}(\eta). 
\end{equation}
\end{prop}

\begin{proof}
We start by letting $A\rtimes G$ act as multipliers on
$\rho^*A\rtimes_\sigma G^H$.  Given $f\in \Gamma_c(G,r^*\mcal{A})$ and
$g\in \Gamma_c(G^H,r^*\rho^*\mcal{A})$ define 
\begin{equation}
M_f(g)([\gamma,\eta]) = \int_G
f(\xi)\alpha_\xi(g([\xi\inv\gamma,\eta])) d\lambda^{r(\gamma)}(\xi).
\end{equation}
The arguments that $M_fg\in\Gamma_c(G^H,r^*\rho^*\mcal{A})$ and that
$(f,g)\mapsto M_fg$ is jointly continuous in the inductive limit
topology are similar to the corresponding arguments for convolution and 
are omitted.  In order to show that $M_f$ defines a multiplier we need
to show that it extends to an adjointable linear operator when we view
$\rho^* A\rtimes_\sigma G^H$ as a right Hilbert module over itself
in the usual fashion.  It is tedious, yet straightforward, to show
that $M$ preserves the following relations for $f,g\in
\Gamma_c(G,r^*\mcal{A})$ and $h,k \in \Gamma_c(G^H,r^*\rho^*\mcal{A})$
\begin{align*}
M_f(h*k) &= M_fh*k, & (M_fh)^* * k &= h^* *(M_{f^*} k), & M_{f*g}h &= M_f
M_g h.
\end{align*}
This shows that, at least on the appropriate dense function algebras,
$M_f$ is an adjointable linear operator and that $M$ is a
$*$-homomorphism on $\Gamma_c(G,r^*\mcal{A})$.  Next, we need to show
that $M$ satisfies a nondegeneracy condition.  Because this portion of
the proof is lengthy and somewhat unenlightening we separate it out
into the following lemma, which we will prove at the end of the
section.  

\begin{lemma}
\label{lem:4}
Functions of the form $M_f g$ are dense in
$\Gamma_c(G^H,r^*\rho^*\mcal{A})$ with respect to the inductive limit
topology. 
\end{lemma}

With this lemma can show that $M_f$ extends to an
adjointable operator on ${\rho^* A\rtimes_\sigma G^H}$ and that $\|M_f
\| \leq \|f\|$ so that as a $*$-homomorphism $M$ extends to
$A\rtimes_\alpha G$.  Suppose $\tau$ is a state on $\rho^*
A\rtimes_\sigma G^H$ and define a pre-inner product on
$\rho^*A\rtimes_\sigma G^H$ via $(g,h)_\tau := \tau(g^* * h)$.  Let
$\mcal{H}_\tau$ denote the resulting Hilbert space and $\mcal{H}_0$
the dense image of $\Gamma_c(G^H,r^*\rho^* \mcal{A})$ in
$\mcal{H}_\tau$.  Given $f\in \Gamma_c(G,r^*\mcal{A})$ we define an
operator $\pi(f)$ for $g\in\mcal{H}_0$ by $\pi(f)g = M_f g$.  
It is a simple matter to show that $\pi(f)$ is well defined and that
$\pi$ defines a homomorphism from
$\Gamma_c(G,r^*\mcal{A})$ into the algebra of linear operators on
$\mcal{H}_0$.  We will now show that we can apply Renault's
Disintegration Theorem.  Since elements of the form $M_fg$ are dense
in the inductive limit topology, it is clear that elements of the form
$\pi(f)h$ are dense in $\mcal{H}_\tau$.  Furthermore, we have asserted
that $M_f g$ is jointly continuous with respect to the inductive limit
topology and it follows that $f\mapsto (\pi(f)g,h)_\tau$ is continuous with
respect to the inductive limit topology. Finally, it is easy to show,
using the fact that $M_f$ is adjointable on $\Gamma_c(G^H,r^*\rho^*
\mcal{A})$ with adjoint $M_{f^*}$ that $(\pi(f)g,h)_\tau =
(g,\pi(f^*)h)_\tau$.  Thus we may apply the Disintegration Theorem
and conclude that $\pi$ extends to a
representation of $A\rtimes G$ on $\mcal{H}_\tau$.  In particular,
this implies that for $f\in \Gamma_c(G,r^*\mcal{A})$ and $g\in
\Gamma_c(G^H,r^*\rho^*\mcal{A})$ we have 
\[
\tau((M_fg)^**(M_fg)) = \|\pi(f)g\|_\tau^2 \leq \|f\|^2 \|g\|_\tau^2 = 
\|f\|^2 \tau(g^* * g) \leq \|f\|^2\|g\|^2.
\]
Since $\tau$ is an arbitrary state, we must have $\|M_f g\| \leq
\|f\| \|g\|$.  It now follows quickly that $M_f$ extends to a
multiplier on $\rho^*A\rtimes_\sigma G^H$ and in turn that $M$ extends
to a $*$-homomorphism from $A\rtimes_\alpha G$ into
$M(\rho^*A\rtimes_\sigma G^H)$.  

At this point we are essentially done.  We cite \cite[Proposition
3.8]{tfb} which states that $\rho^*A\rtimes_\sigma G^H \cong
\mcal{K}(\mcal{Z}_H^G)$ and then extend this isomorphism to
$M(\rho^*A\rtimes_\sigma G^H)$.  Thus we may view $M$ as a
nondegenerate homomorphism 
$\phi : A\rtimes G\rightarrow \mcal{L}(\mcal{Z}_G^H)$.  Furthermore, calculating
that $\phi$ is given by \eqref{eq:4} on elements of the form $g\cdot
z$ where $g\in \Gamma_c(G^H,r^*\rho^*\mcal{A})$ and $z\in\mcal{Z}_0$
is straightforward.  The fact that \eqref{eq:4} holds in general now
follows quickly using \cite[Proposition 6.8]{renaultequiv} and an
inductive limit topology argument.  
\end{proof}

Theorem \ref{thm:induce} is now an immediate consequence of
Proposition \ref{prop:4} and \cite[Proposition 2.66]{tfb}.  We finish
the section with the promised nondegeneracy proof.

\begin{proof}[Proof of Lemma \ref{lem:4}]
Fix $g\in \Gamma_c(G^H,r^*\rho^*\mcal{A})$ and let $L = \supp g$.  
Let $\{a_l\}_{l\in\Lambda}$ be an approximate identity for $A$ and use
\cite[Proposition 6.8]{renaultequiv} to choose, for each $4$-tuple
$(K,U,l,\epsilon)$ consisting of a compact set $K\subset G\unit$, a
conditionally compact neighborhood $U$ of $G\unit$ in $G$, $l\in
\Lambda$, and $\epsilon > 0$, a section $e = e_{(K,U,l,\epsilon)}\in
\Gamma_c(G,r^*\mcal{A})$ such that 
\begin{enumerate}
\item $\supp e \subset U$, 
\item $\int_G \|e(\gamma)\|d\lambda^u(\gamma) \leq 4$ for all $u\in
   K$, and 
\item $\|\int_G e(\gamma)d\lambda^u(\gamma) - a_l(u) \| < \epsilon$
  for all $u\in K$. 
\end{enumerate}
We will show that, when $\kappa = (K,U,l,\epsilon)$ is ordered by
increasing $K$ and $l$ and decreasing $U$ and $\epsilon$, we have
$M_{e_\kappa} g\rightarrow g$ with respect to the inductive limit
topology.  

Fix $\epsilon_1 > 0$ and let $K_1 = \rho(r(L))$.  Next, we
prove that there exists a conditionally compact neighborhood $U_1$ such that
$\xi\in U_1$ implies 
\begin{equation}
\label{eq:6}
\| \alpha_\xi(g([\xi\inv\gamma,\eta])) - g([\gamma,\eta])\| <
\epsilon_1
\end{equation}
for all $[\gamma,\eta]\in G^H$ such that $r(\gamma) = r(\xi)$.  
Suppose not and fix some conditionally compact neighborhood $W$ of
$G\unit$.  
Then for any conditionally compact neighborhood $G\unit \subset U
\subset W$ there exists $\xi_U \in U$ and $[\gamma_U,\eta_U]\in G^H$
such that 
\begin{equation}
\label{eq:5}
\|\alpha_{\xi_U}(g([\xi_U\inv \gamma_U,\eta_U])) -
g([\gamma_U,\eta_U])\| \geq \epsilon_1.
\end{equation}
However, for this to hold one of the terms must be nonzero so that we
must have 
\[
[\gamma_U,\eta_U] \in \widetilde{L} = \{ [\xi\gamma,\eta] : \text{$\xi\in W$,
$[\gamma,\eta]\in L$ and $s(\xi) = r(\gamma)$}\}.
\]
It is not particularly difficult to show that $\widetilde{L}$ is
compact.  Thus, ordering $\{[\gamma_U,\eta_U]\}$ by decreasing $U$, we
may pass to a subnet, relabel, and find new representatives such that
there exists $\gamma,\eta\in G$ with $\gamma_U\rightarrow \gamma$ and
$\eta_U\rightarrow \eta$.  Next, observe that $r(\xi_U) \in
\rho(r(\widetilde{L}))$ for all $U$ so that $\{\xi_U\}$ is contained
in the compact set $W\cap r\inv(\rho(r(\widetilde{L})))$.  Thus we may
pass to another subnet, relabel, and find $\xi\in G$ such that
$\xi_U\rightarrow \xi$.  However, by construction, $\xi$ is contained
in every conditionally compact neighborhood of $G\unit$.  It follows that
$\xi\in G\unit$ and therefore
$\alpha_{\xi_U}(g([\xi_U\inv\gamma_U,\eta_U]))\rightarrow
g([\gamma,\eta])$.  This contradicts \eqref{eq:5} which implies that
\eqref{eq:6} must hold.  

In order to complete our $4$-tuple we will show there exists $l_1\in
\Lambda$ such that $l\geq l_1$ implies
\begin{equation}
\label{eq:7}
\|a_l(r(\gamma))g([\gamma,\eta]) - g([\gamma,\eta])\| < \epsilon_1.
\end{equation}
It clearly suffices to verify \eqref{eq:7} on $L$.  Since $a_l$
factors to an approximate identity on each fibre we have
$a_l(r(\gamma))g([\gamma,\eta])\rightarrow g([\gamma,\eta])$ for each
$[\gamma,\eta]\in L$.  Use the fact that the norm is
upper-semicontinuous to choose for each $[\gamma,\eta]\in L$ some
neighborhood $O_{[\gamma,\eta]}$ of $[\gamma,\eta]$ and some
$b_{[\gamma,\eta]}\in \{a_l\}$ such that 
\begin{equation}
\label{eq:8}
\|b_{[\gamma,\eta]}(r(\xi))g([\xi,\zeta]) - g([\xi,\zeta]) \| <
\frac{\epsilon_1}{3}
\end{equation}
for all $[\xi,\zeta]\in O_{[\gamma,\eta]}$.  Since $L$ is compact, we
can find some finite subcover $\{O_i\}$.  Let $\phi_i \in C_c(G^H)$ be
a partition of unity with respect to $\{O_i\}_{i=1}^N$ so that $\supp\phi_i
\subset O_i$ and $\sum \phi_i([\gamma,\eta]) = 1$ if $[\gamma,\eta]\in
L$.  Next define $h \in \Gamma_c(G^H,r^*\rho^*\mcal{A})$ to be 
$h = \sum_{i=1}^N \phi_i \otimes b_{[\gamma_i,\eta_i]}$.  Then, by
construction, for all $[\xi,\zeta]\in L$ we have 
\begin{equation}
\label{eq:9}
\|h([\xi,\zeta])g([\xi,\zeta]) - g([\xi,\zeta])\| <
\frac{\epsilon_1}{3}.  
\end{equation}
Moving on, we can find $l_1$ such that if $l \geq l_1$ then 
\[
\| a_l b_{[\gamma_i,\eta_i]} - b_{[\gamma_i,\eta_i]} \| <
\frac{\epsilon_1}{3\|g\|_\infty}
\]
for all $i$.  It follows quickly that for all $[\xi,\zeta]\in L$
\begin{equation}
\label{eq:10}
\| a_l (r(\xi))h([\xi,\zeta]) - h(\xi,\zeta) \| <
\frac{\epsilon_1}{3\|g\|_\infty}.
\end{equation}
Finally, using \eqref{eq:9} and \eqref{eq:10} and the fact that
$\|a_l\| < 1$ for all $l$, we compute for $l \geq l_1$ and
$[\xi,\zeta]\in L$
\begin{align*}
\|a_l(r(\xi))g([\xi,\zeta]) - g([\xi,\zeta])\| \leq&
\|a_l(r(\xi))(g([\xi,\zeta]) - h([\xi,\zeta])g([\xi,\zeta]))\|\\
&+\|(a_l(r(\xi))h([\xi,\zeta]) - h([\xi,\zeta]))g([\xi,\zeta])\|\\ 
&+\|h([\xi,\zeta])g([\xi,\zeta]) - g([\xi,\zeta])\| \\
<& \|a_l(r(\xi))\|\frac{\epsilon_1}{3} + \|g\|_\infty
\frac{\epsilon_1}{3\|g\|_\infty} + \frac{\epsilon_1}{3} \leq
\epsilon_1.
\end{align*}

Now, suppose we are given $\epsilon_0 > 0$ and let $\epsilon_1 =
\epsilon_0/(5+\|g\|_\infty)$.  Choose $U_1$ and $l_1$ for $\epsilon_1$
as above.  Then given $e = e_{(K,U,l,\epsilon)}$ with $K_1 \subset K$,
$U\subset U_1$, $l_1 \leq l$, and $\epsilon < \epsilon_1$ we compute
for $[\gamma,\eta]\in G^H$
\begin{align*}
\|M_e g([\gamma,\eta]) - g([\gamma,\eta])\| \leq& \left\|\int_G
  e(\xi)(\alpha_\xi(g([\xi\inv\gamma,\eta]))-g([\gamma,\eta]))d\lambda^{r(\gamma)}(\xi)\right\|
\\
&+ \left\|\left(\int_G
e(\xi)d\lambda^{r(\gamma)}(\xi)-a_l(r(\gamma))\right)g([\gamma,\eta])\right\|\\ 
&+ \|a_l(r(\gamma))g([\gamma,\eta]) - g([\gamma,\eta])\| \\
<& \int_U
\|e(\xi)\|\|\alpha_\xi(g([\xi\inv\gamma,\eta]))-g([\gamma,\eta])\|d\lambda^{r(\gamma)}(\xi)
\\
&+ \epsilon_1\|g([\gamma,\eta])\| + \epsilon_1 \\
<& 4\epsilon_1 + \epsilon_1\|g\|_\infty + \epsilon_1 = \epsilon_0.
\end{align*}
Thus $M_{e_\kappa}g\rightarrow g$ uniformly and it is straightforward
to show that this convergence takes place with respect to the
inductive limit topology. 
\end{proof}

\section{Basic Constructions}
\label{sec:basic-constructions}

\subsection{Transitive Groupoid Crossed Products}

In this (very short) 
section we will use an application of Renault's
Equivalence Theorem to identify the Morita equivalence class of
transitive groupoid crossed products.  This will be used in
Section \ref{sec:stab-t_0-orbits} to prove the main result of the
paper.

\begin{theorem}
\label{thm:transitive}
Suppose $(A,G,\alpha)$ is a groupoid dynamical system and that $G$ is
transitive.  Fix $u\in G\unit$ and let
$\beta$ be Haar measure on $S_u$.  Then $\mcal{X}_0 = C_c(G_u,A(u))$
becomes a pre-$A\rtimes_\alpha G -
A(u)\rtimes_{\alpha|_{S_u}}S_u$-imprimitivity bimodule when equipped
with the following operations.
\begin{align*}
f\cdot z(\gamma) &= \int_G
\alpha_\gamma\inv(f(\eta))z(\eta\inv\gamma)d\lambda^{r(\gamma)}(\eta) \\
z\cdot g(\gamma) &= \int_{S_u}
\alpha_s(z(\gamma s)g(s\inv))d\beta(s) \\
\llangle z,w \rrangle_{A(u)\rtimes S_u}(s) &= 
\int_G z(\eta s\inv)^*\alpha_s(w(\eta))d\lambda_u(\eta)
\\
\lset{A\rtimes G}\llangle z,w \rrangle(\gamma) &= \int_{S_u}
\alpha_{\gamma\zeta s}(z(\gamma\zeta s)w(\zeta s)^*)d\beta(s) 
\end{align*}
Hence, the
completion $\mcal{X}$ of $\mcal{X}_0$ is an imprimitivity bimodule and
$A\rtimes_\alpha G$ is Morita equivalent to $A(u)\rtimes_{\alpha}
S_u$.  
\end{theorem}

\begin{proof}
Let $(A,G,\alpha)$ be as above and fix $u\in G\unit$.  
First recall \cite[Example 2.2]{groupoidequiv} that because $G$ is
a transitive, second countable groupoid
the space $X=G_u = s\inv(u)$ is a $(G,S_u)$-equivalence with respect
to the natural actions of $G$ and $S_u$ on $X$.  
Consider the trivial bundle
$\Upsilon = X\times A(u)$.  This is clearly an usc-bundle whose
space of compactly supported sections may be identified with $\mcal{X}_0 =
C_c(G_u,A(u))$.  Given $\gamma\in X$ we equip $\Upsilon_\gamma = A(u)$
with the $A(r(\gamma))-A(u)$-imprimitivity bimodule structure arising
from the isomorphism $\alpha_\gamma:A(u)\rightarrow A(r(\gamma))$. 
We then define actions of $G$ and $S_u$ on the left and right,
respectively, of $\Upsilon$ by 
\[
\eta\cdot (\gamma,a) = (\eta\gamma,a),\quad\text{and}\quad 
(\gamma,a)\cdot s = (\gamma s, \alpha_s\inv(a))
\]
for $\eta\in G$, $s\in S_u$ and $(\gamma,a)\in \Upsilon$.  It is now
straightforward to show that with these operations $\Upsilon$ is an
equivalence bundle with respect to the dynamical systems
$(A,G,\alpha)$ and $(A(u),\alpha|_{S_u},S_u)$.  The result now follows
from Renault's Equivalence Theorem.
\end{proof}

\subsection{Invariant Ideals}

In this section we would like to see that crossed products behave well
with respect to ideals arising from $G$-invariant sets in $G\unit$.
Given a locally compact $G$-invariant subset $Y\subset G\unit$ we
define the restriction of $G$ to $Y$ to be $G|_Y = r\inv(Y) =
s\inv(Y)$.  It is not difficult to show that $G|_Y$ is a locally
compact Hausdorff groupoid and that the restriction of the Haar system
to $G|_Y$ is a Haar system.  Furthermore, if we have a groupoid
dynamical system $(A,G,\alpha)$ then one can easily show that the
restriction of the action $\alpha$ to $G|_Y$ defines an action of
$G|_Y$ on $A(Y)$.  Hence, once can form the restricted crossed product
$A(Y)\rtimes_\alpha G|_Y$.  
The main goal will be to prove the following result.  

\begin{definition}
\label{def:2}
Given a groupoid dynamical system $(A,G,\alpha)$ and an open
$G$-invariant subset $U$ of $G\unit$ let $\Ex(U)$ be the closure in
$A\rtimes_\alpha G$ of the set 
\[
\{ f\in \Gamma_c(G,r^*\mcal{A}):\supp f\subset G|_U\}.
\]
\end{definition}

\begin{theorem}
\label{thm:invtideal}
Suppose $(A,G,\alpha)$ be a groupoid dynamical system, $U$ is an
open $G$-invariant subset of $G\unit$ and $C$ is the closed $G$-invariant
subset $G\unit\setminus U$.  Then inclusion and restriction extend to
$*$-homomorphisms $\iota : A(U)\rtimes G|_U \rightarrow A\rtimes G$
and $\rho:A\rtimes G \rightarrow A(C)\rtimes G|_C$ respectively.
Furthermore, the following sequence is short exact
\[
\begin{CD} 
0 @>>> A(U)\rtimes G|_U @>\iota>> A\rtimes G @>\rho>> A(C)\rtimes G|_C @>>>0 
\end{CD}
\]
and $\ran \iota = \ker \rho = \Ex(U)$ so that $A(C)\rtimes G|_C$ is
isomorphic to the quotient space $A\rtimes G/\Ex(U)$.  
\end{theorem}

Due to the fact that kernels are ill behaved with respect to
completions, this theorem is deceptively difficult to prove.  
We will study the left hand side of
the short exact sequence first.

\begin{prop}
\label{prop:10}
Given a groupoid dynamical system $(A,G,\alpha)$ and an open
$G$-invariant set $U\subset G\unit$ then the inclusion map 
$\iota:\Gamma_c(G|_U,r^*\mcal{A})\rightarrow \Gamma_c(G,r^*\mcal{A})$
extends to an isomorphism of $A(U)\rtimes_\alpha G|_U$ onto $\Ex(U)$.
Furthermore, $\Ex(U)$ is an ideal in $A\rtimes G$. 
\end{prop}

\begin{proof}
It is straightforward to show that $\Ex(U)$ is an ideal and that
$\iota$ is a homomorphism on $\Gamma_c(G|_U,r^*\mcal{A})$ which is
continuous with respect to the inductive limit topology and which maps onto
a dense subset of $\Ex(U)$.  It now follows from the Disintegration
Theorem that $\iota$ extends to a surjective homomorphism from
$A(U)\rtimes G|_U$ onto $\Ex(U)$.  All that is left is to show that $\iota$
is isometric.  

Suppose $R$ is a faithful representation of $A(U)\rtimes G|_U$ on a
separable Hilbert space $\mcal{H}$.  Let $\mcal{H}_0 = \spn\{R(f)h:
f\in \Gamma_c(G|_U,r^*\mcal{A}),h\in\mcal{H}\}$ and observe that
$\mcal{H}_0$ is dense in $\mcal{H}$.  If $f\in \ran \iota$ and $g\in
\Gamma_c(G,r^*\mcal{A})$ then it is easy to see that
$f*g(\gamma) = 0$ unless $\gamma\in r(\supp f)\subset U$.  Thus
$r(\supp f*g)\subset U$ so that $\supp f*g \subset G|_U$.  In
particular, we can view $f*g$ as a function in
$\Gamma_c(G|_U,r^*\mcal{A})$.  We now define a representation of
$\Gamma_c(G,r^*\mcal{A})$ on $\mcal{H}_0$ via
\begin{equation}
\label{eq:13}
T(f)\sum_{i=1}^n R(g_i)h_i = \sum_{i=1}^n R(f*g_i)h_i.
\end{equation}
We need to show $T$ is well defined.  Using \cite[Proposition
6.8]{renaultequiv} we can find a net $\{e_\kappa\}\subset
\Gamma_c(G|_U,r^*\mcal{A})$ so that $e_\kappa$ is a left approximate
identity with respect to the inductive limit topology.  Suppose
$\sum_i R(g_i)h_i = 0$ and $f\in \Gamma_c(G,r^*\mcal{A})$.  We then
have 
\begin{align*}
\sum_{i=1}^n R(f*g_i)h_i &= \sum_{i=1}^n R(f*\lim_\kappa (e_\kappa*g_i)) \\
&= \sum_{i=1}^n R(\lim_\kappa f*e_\kappa*g_i) h_i  \\
&= \lim_\kappa R(f)R(e_\kappa) \sum_{i=1}^n R(g_i)h_i = 0.
\end{align*}
Next, an elementary computation shows that $T$ 
is a homomorphism into the algebra of linear operators on
$\mcal{H}_0$.  Furthermore, observe that if $f\in
\Gamma_c(G|_U,r^*\mcal{A})$ then $T(\iota(f)) = R(f)$.  We now show
that $T$ satisfies the conditions of the Disintegration Theorem.  It
follows immediately from the fact that $R$ is nondegenerate and
$\{e_\kappa\}$ is a left approximate identity that the set
$\{T(f)k:f\in\Gamma_c(G,r^*\mcal{A}), k\in\mcal{H}_0\}$ has a dense span in
$\mcal{H}$.  The fact that $f\mapsto (T(f)h,k)$ is continuous with
respect to the inductive limit topology for given $h,k\in\mcal{H}_0$ 
follows immediately from the
fact that convolution and $R$ are both continuous with respect to the
inductive limit topology.  Finally, a simple computation
shows that $(T(f)h,k) = (h,T(f^*)k)$ for all
$f\in\Gamma_c(G,r^*\mcal{A})$ and $h,k\in\mcal{H}_0$.  Thus it follows
from the Decomposition Theorem that $T$ is bounded with respect to the
universal norm and extends to a representation of $A\rtimes G$.
Furthermore, since $R=T\circ \iota$ on a dense subset, this identity
holds in general.  Thus, given $f\in A(U)\rtimes G|_U$ we have 
\[
\|f\| = \|R(f)\| = \|T(\iota(f))\| \leq \|\iota(f)\|.
\]
It follows that $\iota$ is isometric and we are done. 
\end{proof}

The complement to Proposition \ref{prop:10} is the following 

\begin{prop}
\label{prop:11}
Suppose $(A,G,\alpha)$ is a groupoid dynamical system and $C$ is a
closed $G$-invariant subset of $G\unit$.  Then the restriction map
$\rho:\Gamma_c(G,r^*\mcal{A})\rightarrow \Gamma_c(G|_C,r^*\mcal{A})$
extends to a surjective homomorphism from $A\rtimes G$ onto
$A(C)\rtimes_\alpha G|_C$.  Furthermore,
$\rho(\Gamma_c(G,r^*\mcal{A}))$ is dense in
$\Gamma_c(G|_C,r^*\mcal{A})$ with respect to the inductive limit
topology. 
\end{prop}

\begin{proof}
This proposition is almost entirely straightforward.  Simple
calculations show that $\rho$ is a homomorphism which is continuous
with respect to the inductive limit topology.  Hence $\rho$ extends to
a bounded homomorphism on $A\rtimes G$.  Standard arguments
using \cite[Proposition C.24]{tfb2} then show that
$\rho(\Gamma_c(G,r^*\mcal{A}))$ is dense with respect to the inductive
limit topology and thus $\rho$ is surjective.  
\end{proof}

Next we prove the following interesting technical lemma.  

\begin{lemma}
\label{lem:21}
If $(A,G,\alpha)$ is a groupoid dynamical system then the function 
\[
u\mapsto \int_G \|f(\gamma)\| d\lambda^u(\gamma)
\]
is upper-semicontinuous for all $f\in\Gamma_c(G,r^*\mcal{A})$.  
\end{lemma}

\begin{proof}
Given $f\in\Gamma_c(G,r^*\mcal{A})$ define
$\lambda(f):G\unit\rightarrow \mathbb{R}$ by  
\[\lambda(f)(u) = \int_G \|f(\gamma)\| d\lambda^u(\gamma).\]
If $\phi\otimes a$ is an elementary tensor in $C_c(G)\odot A$ then 
\[
\lambda(\phi\otimes a)(u) = \int_G |\phi(\gamma)|d\lambda^u(\gamma) \|a(u)\|.
\]
It follows from the continuity of the Haar system and the fact
that $u\mapsto \|a(u)\|$ is upper-semicontinuous that
$\lambda(\phi\otimes a)$ is upper-semicontinuous.  Now suppose
$f\in\Gamma_c(G,r^*\mcal{A})$.  Then there exists a set of
elementary tensors $\{\phi_i^j\otimes a_i^j\}$ such that $k_i = \sum_j
\phi_i^j\otimes a_i^j$ converges to $f$ with 
respect to the inductive limit topology and
therefore with respect to the $I$-norm.  

\begin{remark}
Recall that for groupoid crossed products the $I$-norm is defined for
$f\in \Gamma_c(G,r^*\mcal{A})$ by 
\[
\|f\|_I = \max\left\{ \sup_{u\in G\unit} \int_G \|f(\gamma)\|d\lambda^u(\gamma)
  , \sup_{u\in G\unit} \int_G \|f(\gamma)\|d\lambda_u(\gamma)\right\}.
\]
\end{remark}

Now, $\lambda(k_i) =
\sum_j\lambda(\phi_i^j\otimes a_i^j)$, and it is straightforward to show
that sums of upper-semicontinuous functions are upper-semicontinuous.
Hence, $\lambda(k_i)$ is upper-semicontinuous.  We then conclude
from the computation 
\begin{align*}
\left|\int_G \|f(\gamma)\|d\lambda^u(\gamma) - \int_G
  \|g(\gamma)\|d\lambda^u(\gamma)\right| & \leq
\int_G |\,\|f(\gamma)\|-\|g(\gamma)\|\,| d\lambda^u(\gamma) \\
&\leq \int_G \|f(\gamma)-g(\gamma)\| d\lambda^u(\gamma)\\
&\leq \|f-g\|_I
\end{align*}
that $\lambda(k_i)\rightarrow \lambda(f)$ uniformly.  The
result follows from the fact that uniform limits of
upper-semicontinuous functions are upper-semicontinuous.  
\end{proof}

We now can finish with a proof of Theorem \ref{thm:invtideal} 

\begin{proof}[Proof of Theorem \ref{thm:invtideal}]
All that is left is to show that $\ker\rho = \Ex(U) = \ran\iota$. It
is clear that given $f\in \Gamma_c(G|_U,r^*\mcal{A})$ we have
$\rho\circ\iota(f) = 0$.  Hence $\ran\iota \subset\ker\rho$ and we are
reduced to proving that $\ker\rho \subset \Ex(U)$.  Let $R$ be a
representation of $A\rtimes G$ such that $\ker R=\Ex(U)$.  Now,
suppose we have $f,g\in \Gamma_c(G,r^*\mcal{A})$ such that $\rho(f) =
\rho(g)$. Standard approximation arguments show that $f-g\in
\Ex(U)=\ker R$.  Thus the representation $T$ of
$\Gamma_c(G|_C,r^*\mcal{A})$ given by $T(\rho(f)) = R(f)$ is well
defined.  Furthermore, since $R$ and $\rho$ are homomorphisms, it
follows that $T$ is as well.  We would like to see that $T$ is
$I$-norm decreasing.  

Suppose $f\in \Gamma_c(G,r^*\mcal{A})$ and fix $\epsilon > 0$.  
It follows from Lemma \ref{lem:21} that 
we can find for each $v\in r(\supp f)\cap C$ some relatively
compact open neighborhood $O_v$ of $v$ such that $w\in O_v$ implies 
\begin{equation}
\label{eq:15}
\int_G \|f(\gamma )\|d\lambda^w(\gamma) \leq \int_G \|f(\gamma)\|
d\lambda^v(\gamma) + \epsilon \leq \|\rho(f)\|_I + \epsilon.
\end{equation}
By considering the continuous compactly supported function
$\gamma\mapsto f(\gamma\inv)$ we can, in the same fashion, 
also find for each $v\in r(\supp
f)\cap C$ some relatively compact 
open neighborhood $V_v$ of $v$ such that $w\in V_v$ implies 
\begin{equation}
\label{eq:16}
\int_G \|f(\gamma)\|d\lambda_w(\gamma) \leq \int_G \|f(\gamma)\|
d\lambda_v(\gamma) +\epsilon \leq \|\rho(f)\|_I + \epsilon.
\end{equation}
Since $\{O_v\cap V_v\}$ is an open cover of the compact set $r(\supp
f)\cap C$ there exists some finite subcover $\{O_{v_i}\cap
V_{v_i}\}_{i=1}^N$.  Let $O=\bigcup_i O_{v_i}\cap V_{v_i}$ and observe
that, because the union is finite, $O$ is relatively compact.  Now
choose $\phi \in C_c(G\unit)$ such that $\phi$ is one on $r(\supp
f)\cap C$, zero off $O$, and $0\leq \phi \leq 1$.  Define $g\in
\Gamma_c(G,r^*\mcal{A})$ by $g(\gamma) = \phi(r(\gamma))f(\gamma)$.
It now follows quickly from \eqref{eq:15} and \eqref{eq:16} that 
$\|g\|_I \leq \|\rho(f)\|_I + \epsilon$.  However, $g-f$ is zero on
$C$ by construction so that 
\[
\|T(\rho(f))\| = \|R(f)\| = \|R(g)\| \leq \|g\| \leq \|g\|_I \leq
\|\rho(f)\|_I + \epsilon. 
\]
Since $\epsilon > 0$ is arbitrary, this shows that $T$ is $I$-norm
decreasing, and because $T$ is an $I$-norm decreasing representation it
follows from the Disintegration Theorem that $T$ extends to a
representation on $A(C)\rtimes G|_C$.  Finally, since $T\circ \rho =
R$ on a dense subset, this identity holds everywhere.  Thus $\ker \rho
\subset \ker R = \Ex(U)$ and we are done. 
\end{proof}

\section{Stabilizers and Regular Groupoids}
\label{sec:stab-t_0-orbits}

The goal for this section, and the main result of this paper, is the following

\begin{theorem}
\label{thm:liftstab}
Suppose $(A,G,\alpha)$ is a groupoid dynamical system and that
$G$ is a regular groupoid.  Then every irreducible representation of
$A\rtimes_\alpha G$ is equivalent to one of the form $\Ind_{S_u}^G R$
where $u\in G\unit$ and $R$ is an irreducible representation of
$A(u)\rtimes_\alpha S_u$. 
\end{theorem}

\begin{remark}
\label{rem:1}
Generalizing this result to non-regular groupoids is difficult.  
For group crossed products the result is
known as the Gootman-Rosenberg-Sauvageot theorem and a proof may be
found in \cite[Chapter 9]{tfb2}.  For groupoid $C^*$-algebras the
result is proved in \cite{geneffhan}, and for general groupoid crossed
products the question is still open. 
\end{remark}

\begin{remark}
A concise way of phrasing Theorem \ref{thm:liftstab} would be to say
that every irreducible representation of $A\rtimes G$ is induced from
a stabilizer subgroup.  Unfortunately, this would conflict with
\cite[Definition 8.10]{tfb2}.  The problem lies in the meaning of the
word stabilizers.  In \cite{tfb2} the stabilizers are the stabilizer
subgroups with respect to the action of $G$ on $\Prim A$.  In Theorem
\ref{thm:liftstab} the stabilizers are taken with respect to the
action of $G$ on $G\unit$ and may be larger.  Of course, when $A$ has
Hausdorff spectrum these two notions match up, and in this case it is
not difficult to show that Theorem \ref{thm:liftstab} can be viewed as
a generalization of \cite[Theorem 8.16]{tfb2} using Example
\ref{ex:haus}.  Reconciling these differences in the non-Hausdorff
case is an open question.  
\end{remark}

The key to proving Theorem \ref{thm:liftstab} will be to reduce to the
case where $G\unit/G$ is Hausdorff because in this case we have the
following result. 

\begin{prop}
\label{prop:5}
Suppose $(A,G,\alpha)$ is a groupoid dynamical system and
$G\unit/G$ is Hausdorff.  Then $A\rtimes_\alpha G$ is a
$C_0(G\unit/G)$-algebra with the action $\phi\cdot f(\gamma) =
\phi(G\cdot r(\gamma))f(\gamma)$ for $\phi\in C_0(G\unit/G)$ and $f\in
\Gamma_c(G,r^*\mcal{A})$.  Furthermore, restriction factors to an
isomorphism of $A\rtimes G(G\cdot u)$ onto $A(G\cdot u)\rtimes
G|_{G\cdot u}$.  
\end{prop}

\begin{proof}
Define the action $\Phi(\phi)f = \phi\cdot f$ as above and 
extend $\Phi$ to the unitization $C_0(G\unit/G)^1$ of
$C_0(G\unit/G)$ in the obvious fashion.  It is straightforward
to show that $\Phi(\xi)$ is an
adjointable linear operator on $\Gamma_c(G,r^*\mcal{A})$ 
for each $\xi \in C_0(G\unit/G)^1$
and that $\Phi$ is a $*$-homomorphism.  
We would like to prove that
$\|\Phi(\phi)f\|\leq \|\phi\|_\infty \|f\|$ for all  $\phi\in
C_0(G\unit/G)$ and $f\in\Gamma_c(G,r^*\mcal{A})$.  It will suffice to
show that 
$
\|\phi\|_\infty^2 \langle f, f\rangle - \langle
\Phi(\phi)f,\Phi(\phi)f\rangle \geq 0
$
as elements of $A\rtimes G$, where $\langle f,g\rangle = f^* * g$.
However, using the fact that $\Phi$ is a homomorphism on
$C_0(G\unit/G)^1$, this amounts to showing 
\begin{equation}
\label{eq:11}
\langle \Phi(\|\phi\|_\infty^2 1 - \overline{\phi}\phi)f,f\rangle
\geq 0.
\end{equation}
All elements of the form $\|\phi\|_\infty^2 1 - \overline{\phi}\phi$
are positive in $C_0(G\unit/G)^1$ so there exists $\xi\in
C_0(G\unit/G)^1$ such that $\|\phi\|_\infty^2 1 - \overline{\phi}\phi
= \xi^*\xi$.  Therefore, we have 
\[
\langle \Phi(\|\phi\|_\infty^2 1 - \overline{\phi}\phi)f,f\rangle = 
\langle \Phi(\xi^*\xi)f,f\rangle = \langle \Phi(\xi)f,\Phi(\xi)f\rangle
\geq 0.
\]
It follows that $\Phi(\phi)$ is a bounded operator on
$\Gamma_c(G,r^*\mcal{A})$ with norm less than $\|\phi\|_\infty$ and that
$\Phi(\phi)$ extends to a multiplier on $A\rtimes_\alpha G$.  It is
easy to see that $\Phi$ is a nondegenerate $*$-homomorphism 
and that it maps into the center of the multiplier algebra.  Hence
$A\rtimes G$ is a $C_0(G\unit/G)$-algebra.  

Next, we identify the fibres of $A\rtimes G$.  Fix $u\in G\unit$.
Since $G\unit/G$ is Hausdorff, $G\cdot u$ is closed in $G\unit$.  Let
$O = G\unit\setminus G\cdot u$.  It is clear that $G\cdot u$ and $O$
are both $G$-invariant so that we may apply Theorem
\ref{thm:invtideal} to conclude that the restriction map $\rho$
factors to an isomorphism of $A\rtimes G/\Ex(O)$ with $A(G\cdot u)
\rtimes G|_{G\cdot u}$.  Now define 
\[
I_u = \cspn\{ \phi\cdot f : \phi\in C_0(G\unit/G),
f\in\Gamma_c(G,r^*\mcal{A}), \phi(G\cdot u) = 0\}.
\]
An approximation argument then shows that $I_u = \Ex(O)$.  The result follows
since, by definition, $A\rtimes G(G\cdot u) = A\rtimes G/I_u$. 
\end{proof}

The reason that this is a useful result is that we know a lot about
the fibres when $G\unit/G$ is Hausdorff. 

\begin{corr}
\label{cor:1}
Suppose $(A,G,\alpha)$ is a groupoid dynamical system and that the
orbit space $G\unit/G$ is Hausdorff.  Given $u\in G\unit$ the fibre
$A\rtimes G(G\cdot u)$ is Morita equivalent to $A(u)\rtimes S_u$.
\end{corr}

\begin{proof}
The fibre of $A\rtimes G$ over $G\cdot u$ is, by the previous
proposition, isomorphic to $A(G\cdot u)\rtimes G|_{G\cdot u}$.
However, $G|_{G\cdot u}$ is a {\em transitive} groupoid so the result
follows from Theorem \ref{thm:transitive}.
\end{proof}

This allows us to take the first step in proving the main result. 

\begin{prop}
\label{prop:6}
Suppose $(A,G,\alpha)$ is a groupoid dynamical system and that the
orbit space $G\unit/G$ is Hausdorff.  Then every irreducible
representation of $A\rtimes_\alpha G$ is equivalent to one of the form
$\Ind_{S_u}^G R$ where $u\in G\unit$ and $R$ is an irreducible
representation of $A(u)\rtimes_\alpha S_u$.  
\end{prop}

In the case where $G\unit/G$ is Hausdorff, every irreducible
representation of $A\rtimes G$ is lifted from a fibre $A\rtimes
G(G\cdot u)$, and every irreducible representation of $A\rtimes
G(G\cdot u)$ comes from an irreducible representation of $A(u)\rtimes
S_u$.  It turns out that this two stage description is nothing more
than the usual induction process. 

\begin{lemma}
\label{lem:1}
Suppose $(A,G,\alpha)$ is a groupoid dynamical system and that the
orbit space $G\unit/G$ is Hausdorff.  Given $u\in G\unit$ let $\rho :
A\rtimes G \rightarrow A(G\cdot u)\rtimes G|_{G\cdot u}$ be the
extension of the restriction map on $\Gamma_c(G,r^*\mcal{A})$.
Furthermore, let $\mcal{X}$ be the $A(G\cdot u)\rtimes G|_{G\cdot u} -
A(u)\rtimes S_u$ imprimitivity bimodule from Theorem
\ref{thm:transitive}.  If $R$ is a representation of $A(u)\rtimes S_u$
then $\Ind_{S_u}^G R = \mcal{X}-\Ind(R)\circ \rho$.
\end{lemma}

\begin{proof}
Let $\mcal{Z}_{S_u}^G$ be the imprimitivity bimodule coming from
Proposition \ref{prop:1} and $\mcal{X}$ the imprimitivity bimodule
from Theorem \ref{thm:transitive}.  
This lemma is simple enough once we observe that
$\mcal{X}$ and $\mcal{Z}$ are equal as right Hilbert
modules.  Hence if $R$ acts on $\mcal{H}$ then both 
$\Ind_{S_u}^G R$ and $\mcal{X}-\Ind(R) \circ \rho$ act on
$\mcal{Z}_{S_u}^G \otimes_{A(u)\rtimes S_u} \mcal{H}$.  From here an 
elementary calculation shows that $\Ind_{S_u}^G R(f) =
\mcal{X}-\Ind(R)(\rho(f))$ for $f\in \Gamma_c(G,r^*\mcal{A})$ and this
extends to the entire crossed product by continuity. 
\end{proof}

With this result at are disposal we are mostly done. 

\begin{proof}[Proof of Proposition \ref{prop:6}]
Suppose $(A,G,\alpha)$ is a groupoid dynamical system and that the
orbit space $G\unit/G$ is Hausdorff.  By Proposition \ref{prop:5},
$A\rtimes G$ is a $C_0(G\unit/G)$-algebra.  It then follows from
general $C_0(X)$-algebra theory \cite[Proposition C.5]{tfb2} that any
irreducible representation $T$ is of the form $T = L\circ \rho$ where
$u\in G\unit$, $L$ is an irreducible representation of $A(G\cdot
u)\rtimes G|_{G\cdot u}$ and $\rho$ is the canonical extension of the
restriction map.  However, $A(G\cdot u) \rtimes G|_{G\cdot u}$ is
Morita Equivalent to $A(u)\rtimes S_u$ by Corollary \ref{cor:1} and 
therefore there is an irreducible
representation $R$ of $A(u)\rtimes S_u$ such that $L$ is equivalent to
$\mcal{X}-\Ind R$ \cite[Section 3.3]{tfb}.  
It then follows that $T = L\circ \rho$ and
$\Ind_{S_u}^G R = \mcal{X}-\Ind(R)\circ \rho$ are equivalent and we
are done. 
\end{proof}

As an aside, the following corollary is useful and follows quickly
from Lemma \ref{lem:1}, although we will not detail the proof here.  

\begin{corr}
\label{cor:2}
Suppose $(A,G,\alpha)$ is a groupoid dynamical system and that the
orbit space $G\unit/G$ is Hausdorff.  If $R$ is an irreducible
representation of $A(u)\rtimes S_u$ then $\Ind_{S_u}^G R$ is
irreducible.  Furthermore, if $L$ and $R$ are irreducible
representations of $A(u)\rtimes S_u$ and $\Ind_{S_u}^G R$ is
equivalent to $\Ind_{S_u}^G L$ then $R$ is equivalent to $L$. 
\end{corr}

We will now prove Theorem \ref{thm:liftstab} by extending
Proposition \ref{prop:6} to groupoids which satisfy the Mackey-Glimm
Dichotomy.  As mentioned in the introduction, there are a number of
conditions which are all equivalent and collectively make up the
dichotomy \cite{groupoiddichotomy}. The most useful condition for our
purposes will be the fact that $G$ is regular if and only if
$G\unit/G$ is almost Hausdorff.  

\begin{definition}
\label{def:1}
A, not necessarily Hausdorff, locally compact space $X$ is said to be
{\em almost Hausdorff} if each locally compact subspace $V$ contains
a relatively open nonempty Hausdorff subset. 
\end{definition}

The key fact we will use concerning almost Hausdorff spaces is the
following proposition, which we cite without proof. 

\begin{prop}[{\cite[Lemma 6.3]{tfb2}}] 
\label{prop:7}
Suppose $X$ is a, not necessarily Hausdorff, locally compact space.
Then $X$ is almost Hausdorff if and only if there is an ordinal
$\gamma$ and open sets $\{U_\alpha : \alpha \leq \gamma\}$ such that 
\begin{enumerate}
\item $\alpha < \beta < \gamma$ implies that $U_\alpha \subsetneq
  U_\beta$,
\item $\alpha < \gamma$ implies that $U_{\alpha +1}\setminus U_\alpha$
  is a dense Hausdorff subspace of $X\setminus U_\alpha$, 
\item if $\delta \leq \gamma$ is a limit ordinal then 
\[
U_\delta = \bigcup_{\alpha < \delta} U_\alpha, 
\]
\item $U_0 = \emptyset$ and $U_\gamma = X$.
\end{enumerate}
\end{prop}

The main reason we care about Proposition \ref{prop:7} is that it
allows us to build a composition series so that we may make use of the
following

\begin{lemma}[{\cite[Lemma 8.13]{tfb2}}]
\label{lem:3}
Suppose $\{I_\alpha\}_{\alpha\in\Lambda}$ is a composition series for
a $C^*$-algebra $A$.  Then every irreducible representation $\pi$ of
$A$ lives on a sub-quotient $I_{\alpha+1}/I_\alpha$ for some $\alpha$. 
\end{lemma}

At this point the way forward is clear.  If $G\unit/G$ is almost
Hausdorff then we will build a composition series of crossed products
where the orbit space associated to the sub-quotients is Hausdorff.
This will allow us to use Proposition \ref{prop:6} to prove Theorem
\ref{thm:liftstab}.  The heart of the argument is that
the multistage process of taking a representation of $A(u)\rtimes S_u$
to $A\rtimes G$ through the composition series is
equivalent to the usual induction process.  

\begin{lemma}
\label{lem:2}
Suppose $(A,G,\alpha)$ is a groupoid dynamical system and
that $U\subset V \subset G\unit$ are open $G$-invariant sets.  Then
${A(V\setminus U)\rtimes G|_{V\setminus U}}$ is naturally isomorphic
to the sub-quotient $\Ex(V)/\Ex(U)$.  Furthermore, if $u\in V\setminus
U$ and $R$ is a
representation of $A(u)\rtimes S_u$ then the canonical extension of
$\Ind_{S_u}^{G|_{V\setminus U}}R$ to $A\rtimes G$ is equal to
$\Ind_{S_u}^G R$.  
\end{lemma}

\begin{proof}
Recall that $A(V)\rtimes G|_V$ is
isomorphic to the ideal $\Ex(V)$ in $A\rtimes G$ via the inclusion map
$\iota$. Let $\Ex'(U)$ be the isomorphic image of 
$A(U)\rtimes G|_U$ in $A(V)\rtimes G|_V$ and let $\Ex(U)$ be the image
of $A(U)\rtimes G|_U$ in $A\rtimes G$.  
It is easy to show that $\iota(\Ex'(U))=\Ex(U)$.  It then follows from
Theorem \ref{thm:invtideal} that we have the following chain of isomorphisms
\[
A(V\setminus U)\rtimes G|_{V\setminus U} \cong_{\rho} A(V)\rtimes G|_{V}/
\Ex(U') \cong_\iota \Ex(V)/\Ex(U)
\]
where $\rho$ is given by restriction and $\iota$ is given by
inclusion.

Now suppose $R$ is a representation of $A(u)\rtimes S_u$ on $\mcal{H}$
for $u\in V\setminus U$.  Then $\Ind_{S_u}^{G|_{V\setminus U}} R$ acts
on $\mcal{Z}_{S_u}^{G|_{V\setminus U}}\otimes_{A(u)\rtimes S_u}
\mcal{H}$ and $\Ind_{S_u}^{G} R$ acts
on $\mcal{Z}_{S_u}^{G}\otimes_{A(u)\rtimes S_u}\mcal{H}$.  However,
$\mcal{Z}_{S_u}^{G|_{V\setminus U}}$ and $\mcal{Z}_{S_u}^G$ are 
actually equal as right Hilbert modules so 
that $\Ind_{S_u}^{G|_{V\setminus U}} R$ and $\Ind_{S_u}^G R$ act on
the same space.  Furthermore, both representations are defined on
their respective function algebras via \eqref{eq:1}.  
Thus, after we identify $\Gamma_c(G|_{V\setminus U},r^*\mcal{A})$ as a
subset of $\Gamma_c(G,r^*\mcal{A})$ and untangle the necessary
definitions, it follows that the canonical
extension of $\Ind_{S_u}^{G|_{V\setminus U}} R$ to $A\rtimes G$ is
equal to $\Ind_{S_u}^G R$.  
\end{proof}

We can now prove the main result of the paper. 

\begin{proof}[Proof of Theorem \ref{thm:liftstab}]
Since $G$ is a second countable, locally compact Hausdorff groupoid,
the fact that $G$ is regular implies that $G\unit/G$ is almost
Hausdorff \cite{groupoiddichotomy}.  Therefore there
are open sets $\{V_\beta\}_{0\leq \beta \leq \gamma}$ in $G\unit/G$
satisfying properties (a)-(d) of Proposition \ref{prop:7}.  Let
$q:G\unit\rightarrow G\unit/G$ be the quotient map and $U_\beta =
q\inv(V_\beta)$ for all $0\leq \beta \leq \gamma$.  Then each
$U_\beta$ is an open $G$-invariant subset and we define $I_\beta =
\Ex(U_\beta)$.  It is clear that $I_0 = \{0\}$, $I_\gamma = A\rtimes
G$, and if $\delta < \beta \leq \gamma$ then $I_\delta \subsetneq
I_\beta$.  Finally, suppose $\delta \leq \gamma$ is a limit ordinal and
$f\in \Gamma_c(G|_{U_\delta},r^*\mcal{A})$.  Then $\{U_\beta\}_{\beta
  < \delta}$ is an open cover of $r(\supp f)$.  Since $r(\supp f)$ is
compact there must be a finite subcover, and since the $U_\beta$ are
nested, this implies that there exists $\beta' < \delta$ such that
$r(\supp f) \subset U_{\beta'}$.  Hence $f\in
\Gamma_c(G|_{U_{\beta'}},r^*\mcal{A})\subset I_{\beta'}$.  It now
  follows quickly that 
\[
I_\delta = \overline{\bigcup_{\beta < \delta} I_\beta}.
\]
Thus $\{I_\beta\}$ is a composition series for $A\rtimes G$.  

Suppose $L$ is an irreducible representation of $A\rtimes G$.
Lemma \ref{lem:3} implies that there exists $\beta$ such that $L$
lives on $I_{\beta+1}/I_\beta$.  In other words, there is an
irreducible representation $T$ of $I_{\beta+1}/I_\beta$ such that $L$
is the unique canonical extension of $T$.  Next, Lemma \ref{lem:2}
implies that we can identify $I_{\beta+1}/I_\beta$ with
$A(U_{\beta+1}\setminus U_\beta)\rtimes G|_{U_{\beta+1}\setminus
  U_\beta}$.  Furthermore, $q(U_{\beta+1}\setminus
U_\beta) = V_{\beta+1}\setminus V_\beta$ is Hausdorff, so that by
Proposition \ref{prop:6} there exists $u\in U_{\beta+1}\setminus
U_\beta$ and an irreducible representation $R$ of $A(u)\rtimes S_u$
such that $T$ is equivalent to $R' =
\Ind_{S_u}^{G|_{U_{\beta+1}\setminus U_\beta}} R$.  Hence, the
extension of $T$ to $A\rtimes G$, which is $L$, is equivalent to the
extension of $R'$ to $A\rtimes G$, which is $\Ind_{S_u}^G R$ by Lemma
\ref{lem:2}.  
\end{proof}

We finish by using Lemma \ref{lem:2} to sift out a useful fact. 

\begin{prop}
\label{prop:8}
Suppose $(A,G,\alpha)$ is a groupoid dynamical system and $G$ is
regular.  If $R$ is an irreducible representation of
$A(u)\rtimes S_u$ then $\Ind_{S_u}^G R$ is irreducible.  Furthermore,
if $R$ and $L$ are both irreducible representations of $A(u)\rtimes
S_u$ and $\Ind_{S_u}^G R$ is equivalent to $\Ind_{S_u}^G L$ then $R$
is equivalent to $L$. 
\end{prop}

\begin{proof}
Suppose $R$ is an irreducible representation of $A(u)\rtimes S_u$.
As before, $G\unit/G$ must be locally Hausdorff so that there exists
$\{V_\beta\}$ as in Proposition \ref{prop:7}.  Consider $\Gamma =
\{\beta \leq \gamma : G\cdot u \in V_\beta\}$.  If $\delta = \min
\Gamma$ is a limit ordinal then $G\cdot u \in \bigcup_{\beta < \delta}
V_\beta$.  However, this is a contradiction since it follows that
$G\cdot u \in V_\beta$ for some $\beta < \delta$.  Therefore,
$\delta$ has an immediate predecessor $\sigma$ and  $G\cdot u \in
V_\delta \setminus V_\sigma$.  Suppose $q : G\unit \rightarrow G\unit/G$
is the quotient map, and let
$U_\delta = q\inv(V_\delta)$ and $U_\sigma = q\inv(V_\sigma)$.  Then
$u\in U_\delta \setminus U_\sigma$ and since  $V_\delta\setminus
V_\sigma$ is Hausdorff it follows from Corollary \ref{cor:2} that
$R' = \Ind_{S_u}^{G|_{U_\delta\setminus U_\sigma}} R$ is irreducible.
Hence the extension of $R'$ to $A\rtimes G$ is irreducible and by
Lemma \ref{lem:2} this is exactly $\Ind_{S_u}^G R$.  The rest of the
proposition follows quickly via similar concerns.  
\end{proof}

\bibliographystyle{amsplain}
\bibliography{references}

\providecommand{\bysame}{\leavevmode\hbox to3em{\hrulefill}\thinspace}
\providecommand{\MR}{\relax\ifhmode\unskip\space\fi MR }
\providecommand{\MRhref}[2]{%
  \href{http://www.ams.org/mathscinet-getitem?mr=#1}{#2}
}
\providecommand{\href}[2]{#2}
\begin{thebibliography}{10}

\bibitem{ectcrossed}
Siegfried Echterhoff, \emph{Crossed products with continuous trace}, memoirs of
  the American Mathematical Society \textbf{123} (1996), no.~586, i--viii,
  1--134.

\bibitem{inducprimide}
Siegfried Echterhoff and Dana~P. Williams, \emph{Inducing primative ideals},
  Transactions of the American Mathematical Society \textbf{360} (2008),
  6113--6129.

\bibitem{felldoran}
James M.~G. Fell and Robert~S. Doran, \emph{Representations of {$*$}-algebras,
  locally compact groups, and {Banach $*$-algebraic} bundles}, vol.~2, Academic
  Press Inc., 1988.

\bibitem{mythesis}
Geoff Goehle, \emph{Groupoid crossed products}, Ph.D. thesis, Dartmouth
  College, 2009, arXiv:0905.4681v1.

\bibitem{lscovalg}
Philip Green, \emph{The local structure of twisted covariance algebras}, Acta
  Mathematica \textbf{140} (1978), no.~1, 191--250.

\bibitem{geneffhan}
Marius Ionescu and Dana~P. Williams, \emph{The generalized {Effros-Hahn}
  conjecture for groupoids}, 2008, Indiana University Mathematics Journal, in
  press.

\bibitem{irredreps}
\bysame, \emph{Irreducible representations of groupoid {$C^*$}-algebras},
  Proceedings of the American Mathematical Society \textbf{137} (2009), no.~4,
  1323--1332.

\bibitem{mackey2}
George~W. Mackey, \emph{Induced representations of locally compact groups.
  {I}}, Annals of Mathematics \textbf{55} (1942), no.~2, 101--139.

\bibitem{mackey1}
\bysame, \emph{On induced representations of groups}, American Journal of
  Mathematics \textbf{73} (1951), 576--592.

\bibitem{groupoidequiv}
Paul~S. Muhly, Jean~N. Renault, and Dana~P. Williams, \emph{Equivalence and
  isomorphism for groupoid {$C^*$}-algebras}, Journal of Operator Theory
  \textbf{17} (1987), 3--22.

\bibitem{renaultequiv}
Paul~S. Muhly and Dana~P. Williams, \emph{Renault's equivalence theorem for
  groupoid crossed products}, New York Journal of Mathematics Monographs
  \textbf{3} (2008), 1--83.

\bibitem{pullback}
Iain Raeburn and Dana~P. Williams, \emph{Pull-backs of {$C^*$}-algebras and
  crossed products by certain diagonal actions}, Transactions of the American
  Mathematical Society \textbf{287} (1985), no.~2, 755--777.

\bibitem{tfb}
\bysame, \emph{Morita equivalence and continuous-trace {$C^*$}-algebras},
  Mathematical Surveys and Monographs, vol.~60, American Mathematical Society,
  1998.

\bibitem{groupoiddichotomy}
Arlan Ramsay, \emph{The {Mackey-Glimm} dichotomy for foliations and other
  polish groupoids}, Journal of Functional Analysis \textbf{94} (1990),
  358--374.

\bibitem{frenchrenault}
Jean Renault, \emph{Repr\'esentations des produits crois\`es d'alg\'ebres de
  groupoides}, Journal of Operator Theorey \textbf{18} (1987), 67--97.

\bibitem{rieffelinduce}
Marc~A. Rieffel, \emph{Induced representations of {$C^*$}-algebras}, Advances
  in Mathematics \textbf{13} (1974), 176--257.

\bibitem{rieffelunit}
\bysame, \emph{Unitary representations of group extensions: an algebraic
  approach to the theory of {Mackey} and {Blattner}}, Advances in Mathematics
  Supplementary Studies \textbf{4} (1979), 43--81.

\bibitem{takesaki}
Masamichi Takesaki, \emph{Covariant representations of {$C^*$}-algebras and
  their locally compact automorphism groups}, Acta Mathematica \textbf{119}
  (1967), 273--303.

\bibitem{cpsmooth}
Dana~P. Williams, \emph{The structure of crossed products by smooth actions},
  Journal of the Australian Mathematical Society \textbf{47} (1989), 226--235.

\bibitem{tfb2}
\bysame, \emph{Crossed products of {$C^*$}-algebras}, Mathematical Surveys and
  Monographs, vol. 134, American Mathematical Society, 2007.

\end{thebibliography}

\end{document}